\documentclass[reqno]{amsart}

\usepackage{fullpage,amssymb,dsfont}

\usepackage[all]{xy}

\usepackage{hyperref}

\numberwithin{equation}{section}

\newtheorem{theorem}{Theorem}[section]

\newtheorem{lemma}[theorem]{Lemma}
\newtheorem{proposition}[theorem]{Proposition}
\theoremstyle{remark}
\newtheorem{remark}[theorem]{Remark}

\newtheorem{example}[theorem]{Example}
\theoremstyle{definition}

\newcommand{\bp}{\begin{proof}}
\newcommand{\ep}{\end{proof}}

\mathchardef\mhyph="2D

\newcommand{\C}{\mathbb{C}}
\newcommand{\Z}{\mathbb{Z}}

\newcommand{\eps}{\varepsilon}

\newcommand{\CC}{\mathcal{C}}

\newcommand{\un}{\mathds{1}}

\newcommand\Ad{\operatorname{Ad}}

\newcommand\Irr{\operatorname{Irr}}
\newcommand\Nat{\operatorname{Nat}}
\newcommand\Rep{\operatorname{Rep}}
\newcommand\Tr{\operatorname{Tr}}

\newcommand\kker{\mathfrak{Ker}}

\begin{document}

\title{Probabilistic boundaries of finite extensions of quantum groups}

\author[S. Malacarne]{Sara Malacarne}

\email{saramal@math.uio.no}

\author[S. Neshveyev]{Sergey Neshveyev}

\email{sergeyn@math.uio.no}

\address{Department of Mathematics, University of Oslo, P.O. Box 1053
Blindern, NO-0316 Oslo, Norway}

\date{April 15, 2017; minor changes July 13, 2017}

\maketitle

\begin{abstract}
Given a discrete quantum group $H$ with a finite normal quantum
subgroup $G$, we show that any positive, possibly unbounded,
harmonic function on $H$ with respect to an irreducible invariant
random walk is $G$-invariant. This implies that, under suitable
assumptions, the Poisson and Martin boundaries of~$H$ coincide with
those of $H/G$. A similar result is also proved in the setting of
exact sequences of C$^*$-tensor categories. As an immediate
application, we conclude that the boundaries of the duals of the
group-theoretical easy quantum groups are classical.
\end{abstract}

\bigskip

\section*{Introduction}

The study of probabilistic boundaries of quantum random walks was
initiated in the 90s by Biane~\cite{B1}, who considered random walks
on the duals of compact Lie groups, and by Izumi~\cite{I1}, who
developed the Poisson boundary theory for discrete quantum groups.
The Martin boundary theory for discrete quantum groups was later
developed by Tuset and the second author~\cite{NT1}. Since then the
boundaries have been computed in a number of cases, see
e.g.~\cite{T,VVV,DRVV}. The situation is particularly satisfactory
for amenable quantum groups, where the Poisson boundaries have been
identified for a large and important class of random
walks~\cite{T,NY2}. On the other hand, in the nonamenable case the
duals of free unitary quantum groups remain the main example of a
truly noncommutative computation~\cite{VVV}.

In this note we consider probably the simplest example of discrete
quantum groups that are neither commutative nor cocommutative,
namely, the crossed products
$\ell^\infty(H)=\ell^\infty(\Gamma)\rtimes S$, where $\Gamma$ is a
discrete group and $S$ is a finite group acting on $\Gamma$ by group
automorphisms. They include the duals of the group-theoretical easy
quantum groups recently studied in~\cite{RW}. We show that under
natural assumptions both boundaries of $H$ coincide with the
corresponding classical boundaries of $\Gamma$.

It should be noted that the Poisson boundaries of certain random
walks on crossed products have been already studied in~\cite{I2},
see also~\cite{KNR}, and shown to be isomorphic to crossed products
of Poisson boundaries. There is no contradiction here, we could have
obtained a similar result if we considered degenerate random walks
on $\ell^\infty(\Gamma)\rtimes S$ that are trivial on the $C^*(S)$
part.

In fact, we formulate and prove our results in a more general
setting than that of crossed products. We consider a discrete
quantum group $H$ with a finite normal quantum subgroup $G$, and
show that under suitable assumptions the Poisson and Martin
boundaries of $H$ coincide with those of $H/G$. For Poisson
boundaries of genuine groups this recovers a result of
Kaimanovich~\cite{Kai} obtained as an application of his study of
covering Markov operators. We also obtain similar results for exact
sequences of C$^*$-tensor categories in the framework of categorical
random walks recently developed in~\cite{NY}.

\medskip

\paragraph{\bf Acknowledgement} {\small The research leading to these results
received funding from the European Research Council under the
European Union's Seventh Framework Programme (FP/2007-2013) / ERC
Grant Agreement no. 307663. It was carried out during the authors'
visits to the University of Tokyo and Ochanomizu University. The
preparation of the paper was completed during the second author's
visit to the Texas A\&M University. The authors are grateful to the
staff of these universities for hospitality and to Yasuyuki
Kawahigashi, Makoto Yamashita and Ken Dykema for making these visits
possible. Special thanks go to Makoto Yamashita for inspiring
conversations.}

\bigskip

\section{Invariance of harmonic functions under finite quantum groups}
\label{sec:inv}

Let $H$ be a discrete quantum group, with the von Neumann algebra
$\ell^\infty(H)$ of bounded functions and comultiplication
$\Delta_H\colon \ell^\infty(H)\to
\ell^\infty(H)\bar\otimes\ell^\infty(H)$, see e.g.~\cite{VD}. Recall
that this implies that $\ell^\infty(H)$ is the $\ell^\infty$-direct
sum $\ell^\infty\mhyph\bigoplus_{s\in I}B(H_s)$ of full matrix
algebras, where $I$ is the set of equivalence classes of irreducible
representations of the dual compact quantum group $\hat H$. When $H$ is a genuine group, we have $I=H$ and the comultiplication is given by $\Delta_H(f)(g,h)=f(gh)$ for $f\in\ell^\infty(H)$ and $g,h\in H$.

For a normal state $\phi\in\ell^\infty(H)_*$, consider the
convolution operator
$$
P_\phi\colon\ell^\infty(H)\to\ell^\infty(H), \ \
P_\phi=(\phi\otimes\iota)\Delta_H.
$$
An element $x\in\ell^\infty(H)$ is called $P_\phi$-\emph{harmonic}
if $P_\phi(x)=x$. Note that since $P_\phi$ can be thought of as a
matrix of completely positive maps $B(H_s)\to B(H_t)$, it also makes
sense to talk about positive harmonic elements in the algebra
$\prod_{s\in I}B(H_s)$ of all functions on $H$.

We denote by $\phi^n$ the convolution powers of $\phi$, defined
inductively by $\phi^{n+1}=(\phi^n\otimes\phi)\Delta_H$. Then
$P^n_\phi=P_{\phi^n}$.

\smallskip

Assume now that $G$ is a finite normal quantum subgroup of $H$. This
means that $\ell^\infty(G)$ is finite dimensional and either of the
following equivalent conditions is satisfied~\cite{VV},\cite{KKS}:
\begin{enumerate}
\item[(i)] we are given a surjective normal $*$-homomorphism $\pi\colon \ell^\infty(H)\to\ell^\infty(G)$
respecting the coproducts such that the fixed point algebra $\ell^\infty(G\backslash H)=\{x\mid \alpha_l(x)=1\otimes x\}$ under
the left action $\alpha_l=(\pi\otimes\iota)\Delta_H$ of $G$ on
$\ell^\infty(H)$ coincides with the fixed point algebra
$\ell^\infty(H/G)=\{x\mid \alpha_l(x)=x\otimes 1\}$ under the right action
$\alpha_r=(\iota\otimes\pi)\Delta_H$;
\item[(ii)] we are given an embedding $\C[\hat G]\to \C[\hat H]$ of the Hopf $*$-algebras of regular functions
on the dual compact quantum groups such that $\C[\hat G]$ is
invariant under the left and/or right adjoint action of the Hopf
algebra $\C[\hat H]$ on itself.
\end{enumerate}
Then the quotient discrete quantum group $\Gamma=H/G$ is defined by
letting
$\ell^\infty(\Gamma)=\ell^\infty(H/G)=\ell^\infty(G\backslash H)$
and the coproduct to be the restriction of $\Delta_H$ to
$\ell^\infty(\Gamma)$.



\begin{theorem}\label{thm:invariance}
Assume $H$ is a discrete quantum group, $G$ is a finite normal
quantum subgroup of $H$, and~$\phi$ is a generating normal state on
$\ell^\infty(H)$, meaning that
$\vee_{n\ge1}\operatorname{supp}\phi^{n} =1$. Then any positive,
possibly unbounded, $P_\phi$-harmonic function on $H$ is
$G$-invariant.
\end{theorem}

For genuine discrete groups and bounded harmonic functions this was
proved by Kaimanovich~\cite[Theorem~3.3.1 and Corollary~3]{Kai}
using measure-theoretic methods.

\begin{remark} \label{rem:actions}
The left and right actions of $G$ are both well-defined on the
algebra of all functions on~$H$. In order to see this, let us take
(i) above as our main definition, so we assume that we are given a
surjective normal $*$-homomorphism $\pi\colon
\ell^\infty(H)\to\ell^\infty(G)$ respecting the coproducts. Recall
that $\ell^\infty(H)=\ell^\infty\mhyph\bigoplus_{s\in I}B(H_s)$,
where $I$ is the set of equivalence classes of irreducible
representations of $\hat H$. For every $s\in I$ fix a representative
$U_s\in B(H_s)\otimes C(\hat H)$. Let $p$ be the support of $\pi$,
that is, we have $\ker \pi=(1-p)\ell^\infty(H)$. Then
$p\ell^\infty(H)=\bigoplus_{s\in I_G}B(H_s)$ for a finite subset
$I_G\subset I$, which we can identify with the set of equivalence
classes of irreducible representations of $\hat G$. The tensor
products of the representations $U_s$, $s\in I_G$, decompose
according to the fusion rules of $\hat G$ (and, moreover, the
subcategory of $\Rep \hat H$ generated by these representations can
be identified with $\Rep\hat G$). It follows that by writing $r\sim
s$ if there exists $t\in I_G$ such that $U_s$ is a subrepresentation
of  $U_r\otimes U_t$, we get a well-defined equivalence relation on
$I$ with finite equivalence classes. If $S$ is an equivalence class,
then the action $\alpha_r=(\iota\otimes\pi)\Delta_H$ of $G$ on
$\ell^\infty(H)$ defines by restriction an action on
$\bigoplus_{s\in S}B(H_s)$. From this we see that the action
$\alpha_r$ of $G$ on $\ell^\infty(H)$ extends in an obvious way to
an action on the whole algebra $\prod_{s\in I}B(H_s)$ of functions
on $H$. Similar considerations apply to the left action $\alpha_l$.
Note in passing that using the normality of $G$ it can be checked
that the corresponding equivalence relation on~$I$ is the same as
for $\alpha_r$, see also Section~\ref{sec:cat} for a stronger
statement.
\end{remark}

\bp[Proof of Theorem~\ref{thm:invariance}] Replacing $\phi$ by
$\sum_{n\ge1}\phi^{n}/2^{n}$, which might only increase the set of
harmonic elements, we may assume that $\phi$ is faithful.

Consider first bounded harmonic elements. As in
Remark~\ref{rem:actions}, consider the support $p$ of
$\pi\colon\ell^\infty(H)\to\ell^\infty(G)$. Define the states
$$
\phi_1=\phi(p)^{-1}\phi(p\,\cdot)\ \ \text{and}\ \
\phi_2=\phi(1-p)^{-1}\phi((1-p)\,\cdot).
$$
Denote by $\nu$ the state on $\ell^\infty(G)$ such that
$\phi_1(x)=\nu(\pi(x))$. Then, with $t=\phi(p)$, we have
$$
P_\phi=t
P_{\phi_1}+(1-t)P_{\phi_2}=t(\nu\otimes\iota)\alpha_l+(1-t)P_{\phi_2}.
$$

Consider the conditional expectation
$$
E=(h\otimes\iota)\alpha_l=(\iota\otimes
h)\alpha_r\colon\ell^\infty(H)\to\ell^\infty(H/G),
$$
where $h$ is the Haar state on $C(G)$. It obviously commutes with
$P_{\phi_1}$ and $P_{\phi_2}$, so it suffices to show that there are
no nonzero $P_\phi$-harmonic elements in $\ker E$. Since
$P_{\phi_2}$ is a contraction, for this, in turn, it suffices to
show that the restriction of $P_{\phi_1}$ to $\ker E$ is a strict
contraction.

Since $\ell^\infty(G)$ is finite dimensional and the state $\nu$ is
faithful, there exists $\delta>0$ such that $\nu-\delta h\ge0$. On
$\ker E$ we have
$$
P_{\phi_1}=(\nu\otimes\iota)\alpha_l=((\nu-\delta
h)\otimes\iota)\alpha_l.
$$
This implies that $\|P_{\phi_1}|_{\ker E}\|\le(\nu-\delta
h)(1)=1-\delta$, which is what we need.

\smallskip

Assume now that $a$ is an unbounded positive $P_\phi$-harmonic
element. By adding $1$ we may assume that $a\ge1$. Put $b=E(a)$.
Then $b\ge1$ is again a $P_\phi$-harmonic element. Since it is
$G$-invariant, it is also $P_{\phi_2}$-harmonic. Since $h\ge(\dim
C(G))^{-1}\eps$, where $\eps$ is the counit on $C(G)$, we also have
$b\ge (\dim C(G))^{-1}a$. It follows that the element $c=b^{-1/2}
ab^{-1/2}$ is bounded.

Consider the Doob transform
$$
P^b_\phi=b^{-1/2}P_\phi(b^{1/2}\cdot b^{1/2})b^{-1/2}
$$
of $P_\phi$ defined by $b$. It is a well-defined ucp map on
$\ell^\infty(H)$, and the element $c$ is $P^b_\phi$-harmonic. As the
element~$b$ is $G$-invariant, the operator $P_\phi^b$ commutes with
$E$ and we have $P_{\phi_1}^b=P_{\phi_1}$, so that $P^b_\phi=t
P_{\phi_1}+(1-t)P_{\phi_2}^b$. Now the same argument as in the first
part of the proof applies and we conclude that
$c\in\ell^\infty(H/G)$. Hence $a=b^{ 1/2}cb^{1/2}$ is $G$-invariant.
\ep

Note that, as was already remarked by Kaimanovich~\cite{Kai}, the normality condition in this theorem cannot be dropped. Indeed, otherwise the Poisson boundaries of free products of finite groups would be trivial, which is not true, since such free products are nonamenable except in a few trivial cases.

\smallskip

It is nevertheless tempting to think that Theorem~\ref{thm:invariance} should be
true in a greater generality and that under suitable irreducibility
conditions any harmonic element with respect to a $G$-equivariant
ucp map on a C$^*$-algebra must be $G$-invariant. However, Theorem
4.3.3 in \cite{Kai} shows that the question what such optimal
conditions could be is quite delicate. We will strengthen
Theorem~\ref{thm:invariance} in a somewhat different direction by
showing that the main part of the argument generalizes from finite
to compact quantum groups.

\begin{proposition}\label{prop:invariance}
Let $\alpha\colon A\to C(G)\otimes A$ be an action of a compact
quantum group $G$ with faithful Haar state on a unital C$^*$-algebra
$A$, and $P\colon A\to A$ be a ucp map satisfying the following
properties:
\begin{enumerate}
\item[(i)] $P$ commutes with the unique $G$-equivariant conditional
expectation $E\colon A\to A^G$;
\item[(ii)] $P$ can be written as a convex combination $t P_1+(1-t)P_2$,
$0<t<1$, of two ucp maps such that $P_1=(\nu\otimes\iota)\alpha$ for
some faithful state $\nu$ on $C(G)$.
\end{enumerate}
Then any $P$-harmonic element in $A$ is $G$-invariant.
\end{proposition}

The proof is based on the following lemma.

\begin{lemma}\label{lem:contract}
Let $G$ be a compact quantum group, $\nu$ be a faithful state on
$C(G)$ and $U\in B(H_U)\otimes C(G)$ be a finite dimensional unitary
representation without nonzero invariant vectors. Then $\|(\iota
\otimes \nu)(U)\|<1$.
\end{lemma}

\begin{proof}
Assume $\|(\iota\otimes\nu)(U)\|=1$. Then there exist unit vectors
$\xi,\zeta\in H_U$ such that for the linear functional
$\omega_{\xi,\zeta}=(\cdot\,\xi,\zeta)$ on $B(H_U)$ we have
$(\omega_{\xi,\zeta}\otimes\nu)(U)=1$. In other words, $\nu(x)=1$
for the contraction $x=(\omega_{\xi,\zeta}\otimes\iota)(U)\in C(G)$.
Since $\nu$ is a faithful state, it follows that $x=1$. Applying the
counit $\eps$ on $\C[G]$ to the identity
$(\omega_{\xi,\zeta}\otimes\iota)(U)=1$, we get $(\xi,\zeta)=1$, so
$\xi=\zeta$, and then $U(\xi\otimes1)=\xi\otimes1$. Therefore $\xi$
is an invariant vector, which is a contradiction.
\end{proof}

\bp[Proof of Proposition~\ref{prop:invariance}] As in the proof of
Theorem~\ref{thm:invariance}, it suffices to show that if $x\in\ker
E$ is $P$-harmonic, then $x=0$. Put $B=A^G$ and consider $\ker E$ as
a right pre-Hilbert $B$-module with the inner product $\langle
y,z\rangle =E(y^*z)$. Note that since the Haar state $h$ on $C(G)$
is assumed to be faithful, the conditional expectation
$E=(h\otimes\iota)\alpha\colon A\to B$ is faithful as well. We will
show that $\psi(\langle x,x\rangle)=0$ for any state $\psi$ on $B$.
Assume this is not the case for some $\psi$.

First of all note that by Schwarz's inequality for ucp maps we have
\begin{equation} \label{eq:sch}
\langle P(y),P(y)\rangle=E(P(y)^*P(y))\le
EP(y^*y)=PE(y^*y)=P(\langle y,y\rangle)
\end{equation}
for all $y\in\ker E$. In particular, $$ \psi(\langle
x,x\rangle)\le\psi P(\langle x,x\rangle).
$$
It follows that if we replace $\psi$ by any weak$^*$ limit point of
the states $n^{-1}\sum^{n-1}_{k=0}\psi P^k$ as $n\to\infty$, then
this might only increase the value of $\psi$ at $\langle
x,x\rangle$. Therefore we may assume that $\psi$ is $P$-invariant,
or equivalently, $P_2$-invariant.

Consider now the Hilbert space $H_\psi$ defined by the space $\ker
E$ equipped with the pre-inner product $(y,z)=\psi(\langle z,y
\rangle)$. Then $H_\psi$ becomes a left unitary $C(G)$-comodule. In
other words, there exists a unitary representation $U\in
M(K(H_\psi)\otimes C(G))$ of $G$ such that if $\Lambda_\psi\colon
\ker E\to H_\psi$ denotes the canonical map, then
$$
(\iota\otimes\Lambda_\psi)\alpha(y)=
U^*_{21}(1\otimes\Lambda_\psi(y))
$$
for all $y\in \ker E$, and hence
\begin{equation*}
\Lambda_\psi P_1(y)=(\iota\otimes\nu)(U^*)\Lambda_\psi(y).
\end{equation*}
The representation $U$ has no nonzero invariant vectors, since there
are no nonzero $G$-invariant vectors in $\ker E$. Decomposing $U$
into a direct sum of finite dimensional representations, by
Lemma~\ref{lem:contract} we conclude that
$$
\|(\iota\otimes\nu)(U^*)\xi\|<\|\xi\|
$$
for any nonzero vector $\xi\in H_\psi$. In particular, we have
\begin{equation} \label{eq:contr}
\|\Lambda_\psi P_1(x)\|<\|\Lambda_\psi(x)\|.
\end{equation}

On the other hand, by applying inequality \eqref{eq:sch} to $y=x$
and $P_2$ instead of $P$ and using the $P_2$-invariance of $\psi$,
we get
$$
\|\Lambda_\psi P_2(x)\|\le\|\Lambda_\psi(x)\|.
$$
But together with \eqref{eq:contr} this contradicts the equality
$tP_1(x)+(1-t)P_2(x)=x$. \ep

\bigskip

\section{Probabilistic boundaries}

Assume as in the previous section that $H$ is a discrete quantum
group and $\phi$ is a normal state on $\ell^\infty(H)$. Consider the
space $H^\infty(H;\mu)\subset\ell^\infty(H)$ of bounded
$P_\phi$-harmonic elements. As was shown by Izumi~\cite{I1,I3}, it
is a von Neumann algebra with the new product
$$
x\cdot y=s^*\mhyph\lim_{n\to \infty}P^n_\phi(xy).
$$
It is called (the algebra of bounded measurable functions on) the
\emph{Poisson boundary} of $H$. In this notation the first part of
Theorem~\ref{thm:invariance} states that if $G\subset H$ is a finite
normal quantum subgroup, then for any generating normal state $\phi$
we have
\begin{equation}
H^\infty(H;\phi)=H^\infty(H/G;\bar\phi),
\end{equation}
where $\bar \phi$ is the restriction of $\phi$ to
$\ell^\infty(H/G)\subset\ell^\infty(H)$.

\smallskip

Recall next the definition of the Martin boundary~\cite{NT1}. For
this we have to consider only normal states $\phi$ that are
invariant under the left adjoint action of $\hat H$ on
$\ell^\infty(H)$. In other words, if
$\ell^\infty(H)=\ell^\infty\mhyph\bigoplus_{s\in I}B(H_s)$, then we
consider the states of the form
$$
\phi_\mu=\sum_{s\in I}\mu(s)\phi_s,\ \
\phi_s=\frac{\Tr(\cdot\,\rho^{-1})}{\Tr(\rho^{-1})}\in B(H_s)^*,
$$
where $\mu$ is a probability measure on $I$ and $\rho$ is the
Woronowicz character $f_1$ for $\hat H$. These are precisely the
states~$\phi$ such that the operator $P_{\phi}$ leaves the center
$\ell^\infty(I)$ of $\ell^\infty(H)$ globally invariant, so that it
defines a classical random walk on $I$. To simplify the notation we
will write $P_\mu$ for $P_{\phi_\mu}$.

Assume now that $\mu$ is generating, that is, $\phi_\mu$ is
generating. We also assume that the classical random walk on $I$ is
transient, or equivalently, the Green kernel
$$
G_\mu\colon c_c(H)=\bigoplus_{s\in I}B(H_s)\to\ell^\infty(H),\ \
G_\mu(x)=\sum^\infty_{n=0}P_\mu^n(x),
$$
is well-defined. This is automatically the case if $H$ is not of Kac
type or, more generally, if the quantum dimension function is
nonamenable. Denote by $I_0\in\ell^\infty(H)$ the unit in the matrix
block corresponding to the counit, that is, $I_0$ is characterized
by the property $x I_0=\eps(x)I_0$ for $x\in\ell^\infty(H)$. Then
the function $G_\mu(I_0)\in\ell^\infty(I)$ has no zeros, and the
Martin kernel is defined as the completely positive map
$$
K_\mu\colon c_c(H)\to\ell^\infty(H),\ \ K_\mu(x)=G_\mu(I_0)^{-1}G_\mu(x).
$$

The antipode defines an involution $s\mapsto \bar s$ on the set $I$.
For a measure $\mu$ on $I$, we denote by $\check\mu$ the measure
such that $\check\mu(s)=\mu(\bar s)$. If $\mu$ is transient, then
$\check\mu$ is transient as well.

Consider the C$^*$-subalgebra of $\ell^\infty(H)$ generated by
$c_0(H)=c_0\mhyph\bigoplus_{s\in I}B(H_s)$ and
$K_{\check\mu}(c_c(H))$. Its quotient by $c_0(H)$ is called the
\emph{Martin boundary} of $H$, and we denote it by $C(\partial
H_{M,\mu})$.

\begin{theorem} \label{thm:martin}
Let $H$ be a discrete quantum group,
$\ell^\infty(H)=\ell^\infty\mhyph\bigoplus_{s\in I}B(H_s)$. Assume
$G\subset H$ is a finite normal quantum subgroup. Consider the
quotient quantum group $H/G$,
$\ell^\infty(H/G)=\ell^\infty\mhyph\bigoplus_{t\in\bar I}B(H_t)$.
Then for any transient generating finitely supported probability
measure $\mu$ on $I$, the embedding
$\ell^\infty(H/G)\hookrightarrow\ell^\infty(H)$ induces an
isomorphism $C(\partial (H/G)_{M,\bar\mu})\cong C(\partial
H_{M,\mu})$, where $\bar\mu$ is the measure on $\bar I$
characterized by $\phi_{\bar\mu}=\phi_\mu|_{\ell^\infty(H/G)}$.
\end{theorem}

\bp As in Section~\ref{sec:inv}, consider the support $p$ of the
homomorphism $\pi\colon\ell^\infty(H)\to\ell^\infty(G)$ and the
$G$-equivariant conditional expectation $E=(\iota\otimes
h\pi)\Delta_H\colon\ell^\infty(H)\to\ell^\infty(H/G)$. We have
$c_0(H)\cap\ell^\infty(H/G)=c_0(H/G)$, and $E$ maps $c_c(H)$ onto
$c_c(H/G)$. We claim that
$$
K_{\check\mu}(x) - nK_{\check{\bar\mu}}(E(x))\in c_0(H)\ \ \text{for any}\ \ x\in c_c(H),
$$
where $n=\dim C(G)$. This obviously proves the theorem.

Note that $p\in\ell^\infty(H/G)$ is exactly the projection defining the counit, so
\begin{equation}\label{eq:m1}
K_{\check{\bar\mu}}(x)=G_{\check\mu}(p)^{-1}G_{\check\mu}(x)
=K_{\check\mu}(p)^{-1}K_{\check\mu}(x)\ \ \text{for}\ \ x\in c_c(H/G)\subset c_c(H).
\end{equation}
Note also that
\begin{equation}\label{eq:m2}
E(I_0)=\frac{1}{n}p,
\end{equation}
since $\pi(I_0)$ is the projection defining the counit on $C(G)$ and
therefore we have $h(\pi(I_0))=1/n$, which can be seen by recalling that if we identify the C$^*$-algebra $C(G)$ with a direct sum of full matrix algebras $B(H_i)$, then the Haar state is the sum of the traces $\frac{\dim H_i}{n}\Tr_{B(H_i)}$.

Next, let $\psi$  be a right invariant Haar weight on
$\ell^\infty(H)$. By \cite[Theorem 3.3]{NT1}, for any state $\omega$
on $C(\partial H_{M,\mu})$ there exists a unique, possibly
unbounded, positive $P_\mu$-harmonic function $x_\omega$ on $H$ such
that
\begin{equation*}\label{kernels}
\psi(y\sigma^\psi_{-i/2}(x_\omega))=\omega (K_{\check \mu}(y))\ \
\text{for all}\ \ y\in c_c(H),
\end{equation*}
where $\sigma^\psi_t= \Ad \rho^{-it}$ is the modular group of
$\psi$. By Theorem~\ref{thm:invariance} we have
$E(x_\omega)=x_\omega$. Note also that $E$ commutes with
$\sigma^\psi_t$. It follows that if $y\in c_c(H)\cap \ker E$, then
$$
\omega( K_{\check\mu}(y))=\psi(y\sigma^\psi_{-i/2}(x_\omega))
=(\psi*
h\pi)(y\sigma^\psi_{-i/2}(x_\omega))=\psi(E(y\sigma^\psi_{-i/2}(x_\omega)))=0.
$$
Since this is true for any $\omega$, we conclude that
\begin{equation} \label{eq:m3}
K_{\check\mu}(y)\in c_0(H)\ \ \text{for all}\ \ y\in c_c(H)\cap\ker E.
\end{equation}

Now, \eqref{eq:m2} and \eqref{eq:m3} show that
$$
n1-  K_{\check\mu}(p)=K_{\check\mu}(nI_0 - p)\in c_0(H).
$$
Using \eqref{eq:m1} and again \eqref{eq:m3}, for any $x\in c_c(H)$
we then get the following equalities modulo $c_0(H)$:
$$
K_{\check\mu}(x)=K_{\check\mu}(E(x))=K_{\check\mu}(p)K_{\check{\bar \mu}}(E(x))
=nK_{\check{\bar\mu}}(E(x)),
$$
which proves our claim.
\ep

Let us give a simple class of noncommutative examples where the
above results can be applied.

\begin{example}
Let $\Gamma$ be a discrete group and $S$ be a finite group acting on
$\Gamma$ by group automorphisms $\gamma\mapsto s.\gamma$. We can
then define a discrete quantum group $H$ with the algebra of bounded
measurable functions $\ell^\infty(H)=\ell^\infty(\Gamma)\rtimes S$
and the coproduct extending the usual coproducts on
$\ell^\infty(\Gamma)$ and $C^*(S)$. This is a quantum group unless
$S$ is an abelian group acting trivially on $\Gamma$. The dual
$G=\hat S$ is a normal quantum subgroup of $H$, with the structure
homomorphism $\pi\colon\ell^\infty(H)\to\ell^\infty(G)=C^*(S)$ given
by $f\lambda_s\mapsto f(e)\lambda_s$, and we have $H/G=\Gamma$. We
thus see that, under suitable assumptions, the Poisson and Martin
boundaries of $H$ coincide with the corresponding classical
boundaries of $\Gamma$.

Note that the $\hat H$-invariant normal states $\phi_\mu$ are
exactly the normal tracial states on $\ell^\infty(\Gamma)\rtimes S$.
It is not difficult to show, see~\cite{N} for a more general
statement, that such traces are given by $S$-invariant probability
measures $\bar\mu$ on $\Gamma$ and tracial states $\tau_\gamma$ on
$C^*(S_\gamma)$, where $S_\gamma\subset S$ is the stabilizer of
$\gamma\in\Gamma$, such that
$\tau_{s.\gamma}(\lambda_s\cdot\lambda_s^*)=\tau_{\gamma}$. Namely,
the trace $\phi_\mu$ corresponding to a pair $(\bar\mu,
(\tau_\gamma)_{\gamma\in\Gamma})$ is given by
$$
\phi_\mu(f\lambda_s)=\sum_{\gamma:\ s\in
S_\gamma}\bar\mu(\gamma)f(\gamma)\tau_\gamma(\lambda_s).
$$
It can be checked that the trace $\phi_\mu$ is faithful if and only
if $\operatorname{supp}\bar\mu=\Gamma$ and every trace $\tau_\gamma$
is faithful. It follows then that, more generally, the trace
$\phi_\mu$ is generating if the set of elements
$\gamma\in\operatorname{supp}\bar\mu$ such that $\tau_\gamma$ is
faithful generates $\Gamma$ as a semigroup. It is clear also that
$\phi_\mu$ is transient if and only if $\bar\mu$ is transient. This
allows one to construct many examples where the assumptions of
Theorem~\ref{thm:martin} are satisfied.

This class of discrete quantum groups $H$ includes the duals of the
group-theoretical easy quantum groups~\cite{RW}. These duals are
obtained by taking $\Gamma$ to be a quotient of $(\Z/2\Z)^{*n}$ and
$S$ to be the symmetric group $S_n$ acting on $\Gamma$ by permuting
the generators.
\end{example}

\bigskip

\section{Categorical analogue} \label{sec:cat}

In this section we will prove an analogue of
Theorem~\ref{thm:invariance} for C$^*$-tensor categories. Our
conventions are the same as in \cite{NT}. Briefly, we assume that
the categories that we consider are small, closed under subobjects
and finite direct sums, and the tensor units are simple unless
explicitly stated otherwise. We also assume that the categories are
strict. We denote the morphisms sets in a category $\CC$ by
$\CC(U,V)$ and write $\CC(U)$ for $\CC(U,U)$.

Recall that for an object $U$ in a C$^*$-tensor category $\CC$, the
conjugate, or dual, object is an object $\bar U$ such that there
exist morphisms $R\colon\un\to \bar U\otimes U$ and $\bar
R\colon\un\to U\otimes\bar U$ solving the conjugate equations
$$
(R^*\otimes\iota)(\iota\otimes\bar R)=\iota,\ \ (\bar
R^*\otimes\iota)(\iota\otimes R)=\iota.
$$
The minimum of the numbers $\|R\|\,\|\bar R\|$ over all solutions is
called the dimension of $U$. We denote the dimension by $d(U)$. A
solution $(R,\bar R)$ is called standard if $\|R\|=\|\bar
R\|=d(U)^{1/2}$. A category in which every object has a dual object
is called rigid.

Fixing a standard solution $(R_U,\bar R_U)$ of the conjugate
equations for every object $U$ we can define maps
$$
\CC(U\otimes V,U\otimes W)\to \CC(V,W), \ \ T\mapsto
(R^*_U\otimes\iota)(\iota\otimes T)(R_U\otimes\iota),
$$
which are denoted by $\Tr_U\otimes\iota$ and called partial
categorical traces. They are independent of the choice of standard
solutions.

\smallskip

Recall next the notion of a harmonic natural
transformation~\cite{NY}. Fix objects $U$ and $V$ and consider the
space $\Nat_b(\iota\otimes U,\iota\otimes V)$ of bounded natural
transformation between the functors $\iota\otimes U$ and
$\iota\otimes V$, that is, a uniformly bounded collection
$\eta=(\eta_X)_X$ of natural in $X$ morphisms $X\otimes U\to
X\otimes V$. For every object $W$ we then define an operator $P_W$
on $\Nat_b(\iota\otimes U,\iota\otimes V)$ by
$$
P_W(\eta)_X=d(W)^{-1}(\Tr_W\otimes\iota)(\eta_{W\otimes X}).
$$
Consider the set $\Irr(\CC)$ of isomorphism classes of simple
objects in $\CC$, and choose for every $s\in\Irr(\CC)$ a
representative $U_s$. For a probability measure $\mu$ on $\Irr(\CC)$
we put
$$
P_\mu=\sum_s\mu(s)P_{U_s}.
$$
A natural transformation $\eta\in\Nat_b(\iota\otimes U,\iota\otimes
V)$ is called $P_\mu$-\emph{harmonic} if $P_\mu(\eta)=\eta$. If
$U=V$ then it makes sense to also talk about unbounded positive
$P_\mu$-harmonic natural transformations.

Define convolution of measures on $\Irr(\CC)$ by
$$
(\nu*\mu)(t)=\sum_{s,r}\nu(s)\mu(r)m^t_{sr}\frac{d(U_t)}{d(U_s)
d(U_r)},
$$
where $m^t_{sr}$ is the multiplicity of $U_t$ in $U_s\otimes U_r$.
Then $P_\mu P_\nu=P_{\nu*\mu}$. We write $\mu^n$ for the $n$th
convolution power of $\mu$.

\smallskip

We next recall a few notions from~\cite{BN}, with obvious
modifications needed in our C$^*$-setting. Let $F\colon\CC\to\CC''$
be a unitary tensor functor between C$^*$-tensor categories. The
functor $F$ is called \emph{normal}, if for every object $U\in\CC$
there exists a subobject $U_0$ such that $F(U_0)$ is the largest
subobject of $F(U)$ which is trivial, that is, isomorphic to $\un^n$
for some $n$. We then denote by $\kker_F\subset\CC$ the full
subcategory consisting of objects $U$ such that $F(U)$ is trivial. A
sequence
$$
\CC'\xrightarrow{i}\CC\xrightarrow{F} \CC''
$$
of unitary tensor functors is called \emph{exact}, if
\begin{enumerate}
\item[(a)] $F$ is dominant, that is, every object of
$\CC''$ is a subobject of $F(U)$ for some $U\in\CC$;
\item[(b)] $F$ is normal;
\item[(c)] $i$ defines an equivalence between $\CC'$ and $\kker_F$.
\end{enumerate}

\smallskip

Given a C$^*$-tensor category $\CC$ and a full C$^*$-tensor
subcategory $\CC'\subset\CC$, let us say that $\CC'$ is
\emph{normal} if there exists an exact sequence $
\CC'\xrightarrow{i}\CC\xrightarrow{F} \CC''$, where $i$ is the
embedding functor. Let us also say that a natural transformation
$\eta$ between the functors $\iota\otimes U$ and $\iota\otimes V$ on
$\CC$ is $\CC'$-\emph{invariant}, if
$$
\eta_{X\otimes Y}=\iota_X\otimes\eta_Y\ \ \text{for all}\ \
X\in\CC'\ \ \text{and}\ \ Y\in\CC.
$$

We are now ready to formulate an analogue of
Theorem~\ref{thm:invariance}.

\begin{theorem} \label{thm:invariance2}
Let $\CC$ be a rigid C$^*$-tensor category, $\CC'\subset\CC$ be a
finite normal C$^*$-tensor subcategory, and~$\mu$ be a generating
probability measure on $\Irr(\CC)$, meaning that
$\cup_{n\ge1}\operatorname{supp}\mu^n=\Irr(\CC)$. Then any positive,
possibly unbounded, $P_\mu$-harmonic natural transformation
$\eta\colon\iota\otimes U\to\iota\otimes U$ is $\CC'$-invariant.
\end{theorem}

Note that by \cite[Theorem~4.1]{NY2} this theorem generalizes
Theorem~\ref{thm:invariance} for $\hat H$-invariant $\phi$, but not
for arbitrary states, so formally these two results are independent.
Not surprisingly, the proofs are nevertheless similar. But before we turn to the
proof we need to formulate $\CC'$-invariance in a more analytic way.

Define a probability measure $h'$ on $\Irr(\CC')$ by
$$
h'(s)=\frac{d(U_s)^2}{d(\CC')},\ \ \text{where}\ \
d(\CC')=\sum_{t\in\Irr(\CC')}d(U_t)^2.
$$
It is known, and is easy to see using multiplicativity and
additivity of the dimension function, that for any probability
measure $\nu$ on $\Irr(\CC')$ we have
\begin{equation}\label{eq:h1}
\nu*h'=h'*\nu=h'.
\end{equation}
Since we can identify $\Irr(\CC')$ with a subset of $\Irr(\CC)$, we
can also view $h'$ as a measure on $\Irr(\CC)$.

\begin{lemma}\label{lem:invariance}
For a full finite rigid C$^*$-tensor subcategory $\CC'$ of a rigid
C$^*$-tensor category $\CC$, a natural transformation
$\eta\colon\iota\otimes U\to\iota\otimes U$ is $\CC'$-invariant if
and only if $P_{h'}(\eta)=\eta$.
\end{lemma}

\bp If $\eta$ is $\CC'$-invariant, then obviously $P_\nu(\eta)=\eta$
for any probability measure $\nu$ on $\Irr(\CC')$, in particular,
for $h'$. Conversely, assume $P_{h'}(\eta)=\eta$. It suffices to
show that $\eta_X=\iota_X\otimes\eta_\un$ for all $X\in\CC'$, since
by applying this statement to the natural transformation
$(\eta_{Z\otimes Y})_{Z\in\CC}\colon\iota\otimes Y\otimes
U\to\iota\otimes Y\otimes U$ (which was denoted by $\iota_Y\otimes
\eta$ in \cite{NY}) we then get $\eta_{X\otimes
Y}=\iota_X\otimes\eta_Y$ for all $X\in\CC'$, as required. For this,
in turn, consider $\iota\otimes U$ as a functor $\CC'\to\CC$. Then
$(\eta_X)_{X\in\CC'}$ is an endomorphism of this functor, while
$P_{h'}$ can be considered as an operator on the space of such
endomorphisms. By \cite[Proposition~2.4]{NY}, when $\CC=\CC'$, the
subspace of $P_\nu$-invariant endomorphisms consists of the elements
$(\iota_X\otimes T)_{X\in\CC'}$, with $T\in\CC(U)$, for any
generating probability measure $\nu$ on $\Irr(\CC')$. The same proof
works in general, so $\eta_X=\iota_X\otimes\eta_\un$ for all
$X\in\CC'$. \ep

Assume now that we are in the setting of
Theorem~\ref{thm:invariance2} and consider the corresponding exact
sequence $ \CC'\to\CC\xrightarrow{F} \CC''$. Since $F(X)$ is trivial
for every $X\in\CC'$, $F|_{\CC'}$ can be considered as a unitary
fiber functor. This already implies that the dimension function on
$\CC'$ is integral and that $\CC'$ can be identified with $\Rep G$
for a finite quantum group $G$. Consider the object
$$
A=\bigoplus_{s\in\Irr(\CC')}U_s^{d(U_s)}\in\CC'.
$$
It is shown in \cite[Section~5.2]{BN} that $A$ admits the structure
of a commutative central algebra in $\CC$. In particular, $A\otimes
Y\cong Y\otimes A$ for all $Y\in\CC$, which implies that
\begin{equation}\label{eq:h2}
\nu*h'=h'*\nu
\end{equation}
for any probability measure $\nu$ on $\Irr(\CC)$.

\bp[Proof of Theorem~\ref{thm:invariance2}] The proof goes along the
same lines as that of Theorem~\ref{thm:invariance}. We will only
consider the case of bounded natural transformations, the general
case is dealt with similarly to the second part of that proof.

We may assume that $\operatorname{supp}\mu=\Irr(\CC)$. We can then
write $\mu$ as a convex combination $t\mu_1+(1-t)\mu_2$ of two
measures, with $0<t<1$, $\mu_1$ supported on $\Irr(\CC')$ and
$\mu_2$ on $\Irr(\CC)\setminus\Irr(\CC')$. By \eqref{eq:h1} and
Lemma~\ref{lem:invariance}, the operator $E=P_{h'}$ defines a
projection onto the space of $\CC'$-invariant natural
transformations. By \eqref{eq:h2} this projection commutes with
$P_{\mu_1}$ and $P_{\mu_2}$. Therefore it suffices to show that the
restriction of $P_{\mu_1}$ to $\ker E$ is a strict contraction.
Since $\operatorname{supp}\mu_1=\Irr(\CC')$, there exists $\delta>0$
such that $\mu_1-\delta h'$ is a positive measure. Since
$P_{\mu_1}=P_{\mu_1-\delta h'}$ on $\ker E$, it follows then that
the norm of the restriction of $P_{\mu_1}$ to $\ker E$ is bounded by
$(\mu_1-\delta h')(\Irr(\CC))=1-\delta$. \ep

\begin{remark}
Theorem~\ref{thm:invariance} can be formulated by saying that, under
its assumptions, any positive harmonic function on $H$ arises from
that on $H/G$. In a similar way Theorem~\ref{thm:invariance2}
implies that any positive harmonic natural transformation
$\iota\otimes U\to\iota\otimes U$ of functors on $\CC$ arises from a
natural transformation $\iota\otimes F(U)\to\iota\otimes F(U)$ of
functors on $\CC''$. In other words, we claim that if an
endomorphism $\eta=(\eta_X)_{X\in\CC}$ of $\iota\otimes U$ is
$\CC'$-invariant, then the collection of morphisms
$$
F(X)\otimes F(U)\xrightarrow{F_2} F(X\otimes U)\xrightarrow{
F(\eta_X)} F(X\otimes U)\xrightarrow{F_2^{-1}}F(X)\otimes F(U)
$$
defines, necessarily uniquely, an endomorphism of $\iota\otimes
F(U)$. Indeed, by \cite[Corollary~5.8]{BN}, the category~$\CC''$ can
be identified with the category $A\mhyph\operatorname{mod}_\CC$ of
left $A$-modules in $\CC$ and then the functor $F$ is given by
$F(X)=A\otimes X$. Note that any left $A$-module can also be
considered as an $A$-bimodule by the commutativity of the central
algebra $A$, and this turns $A\mhyph\operatorname{mod}_\CC$ into a
tensor category with the tensor product $\otimes_A$. It follows that
our claim is equivalent to the statement that for any
$\CC'$-invariant $\eta$ the morphisms $\iota_A\otimes\eta_X\colon
A\otimes X\otimes U\to A\otimes X\otimes U$ are natural with respect
to the $A$-module morphisms $A\otimes X\to A\otimes Y$. But this is
clear, as $\iota_A\otimes\eta_X=\eta_{A\otimes X}$.
\end{remark}

\bigskip


\bigskip

\begin{thebibliography}{99}

\bibitem{B1}
Ph. Biane, \textit{Th\'eor\`eme de Ney-Spitzer sur le dual de
SU(2)}, Trans. Amer. Math. Soc. {\bf 345} (1994), no.~1, 179--194.

\bibitem{BN}
A. Brugui\`eres and S. Natale, {\em Exact sequences of tensor
categories}, Int. Math. Res. Not. IMRN {\bf 2011}, no.~24,
5644--5705.

\bibitem{DRVV}
A. De Rijdt and N. Vander Vennet, \textit{Actions of monoidally
equivalent compact quantum groups and applications to probabilistic
boundaries}, Ann. Inst. Fourier (Grenoble) {\bf 60} (2010), no.~1,
169--216.

\bibitem{I1}
M. Izumi, \textit{Non-commutative Poisson boundaries and compact
quantum group actions}, Adv. Math \textbf{169} (2002), no.~1, 1--57.

\bibitem{I2}
M. Izumi, \textit{Non-commutative Poisson boundaries}, in: Discrete
geometric analysis, Contemp. Math. \textbf{347}, Amer. Math. Soc.,
Providence, RI (2004), 69--81.

\bibitem{I3}
M. Izumi, \textit{$E_0$-semigroups: around and beyond Arveson's
work}, J. Operator Theory {\bf 68} (2012), no.~2, 335--363.

\bibitem{Kai}
V. Kaimanovich, \textit{The Poisson boundary of covering Markov
operators}, Israel J. Math. {\bf 89} (1995), no.~1-3, 77--134.

\bibitem{KKS}
M. Kalantar, P. Kasprzak and A. Skalski, \textit{Open quantum
subgroups of locally compact quantum groups}, Adv. Math. {\bf 303}
(2016), 322--359.

\bibitem{KNR}
M. Kalantar, M. Neufang and Z.-J. Ruan, \textit{Realization of
quantum group Poisson boundaries as crossed products}, Bull. Lond.
Math. Soc. {\bf 46} (2014), no.~6, 1267--1275.

\bibitem{N}
S. Neshveyev, \textit{KMS states on the $\textrm{C}^*$-algebras of
non-principal groupoids}, J. Operator Theory \textbf{70} (2013),
no.~2, 513--530.

\bibitem{NT1}
S. Neshveyev and L. Tuset, \textit{The Martin boundary of a discrete
quantum group}, J. Reine Angew. Math. \textbf{568} (2004), 23--70.

\bibitem{NT}
S. Neshveyev and L. Tuset, \textit{Compact quantum groups and their
representation categories}, Cours Sp\'ecialis\'es. Soci\'et\'e
Math\'ematique de France \textbf{20}, Paris (2013).

\bibitem{NY2}
S. Neshveyev and M. Yamashita, \textit{Categorical duality for
Yetter-Drinfeld algebras}, Doc. Math. {\bf 19} (2014), 1105--1139.

\bibitem{NY}
S. Neshveyev and M. Yamashita, \textit{Poisson boundaries of
monoidal categories}, preprint arXiv: 1405.6572 [math.OA], to appear
in Ann. Sci. \'Ec. Norm. Sup\'er.

\bibitem{RW}
S. Raum and M. Weber, \textit{Easy quantum groups and quantum
subgroups of a semi-direct product quantum group}, J. Noncommut.
Geom. \textbf{9} (2015), no.~4, 1261--1293.

\bibitem{T}
R. Tomatsu, \textit{A characterization of right coideals of quotient
type and its application to classification of Poisson boundaries},
Comm. Math. Phys. {\bf 275} (2007), no.~1, 271--296.

\bibitem{VV}
S. Vaes and L. Vainerman, \textit{On low-dimensional locally compact
quantum groups}, in: Locally compact quantum groups and groupoids
(Strasbourg, 2002), IRMA Lect. Math. Theor. Phys., {\bf 2}, de
Gruyter, Berlin (2003), 127--187.

\bibitem{VVV}
S. Vaes and N. Vander Vennet, \textit{Poisson boundary of the
discrete quantum group $\hat A_u(F)$}, Compos. Math. {\bf 146}
(2010), no.~4, 1073--1095.

\bibitem{VD}
A. Van Daele, \textit{Discrete quantum groups}, J. Algebra
\textbf{180} (1996), no.~2, 431--444.

\end{thebibliography}
\end{document}